\documentclass[12pt]{amsart}
\usepackage{amsmath,amssymb}%,showkeys}

\newtheorem{theorem}{Theorem}[section]
\newtheorem{lemma}[theorem]{Lemma}
\newtheorem{claim}[theorem]{Claim}

\newtheorem{proposition}[theorem]{Proposition}

\theoremstyle{definition}

\theoremstyle{remark}
\newtheorem{remark}[theorem]{Remark}

\DeclareMathOperator{\Vol}{Vol} \DeclareMathOperator{\dist}{dist}
 
\DeclareMathOperator{\Hess}{Hess} 
 
\DeclareMathOperator{\diam}{diam} 
\DeclareMathOperator{\V}{V}

\author{Gang Tian}
\thanks{The first author is supported in part by National Science Foundation grants DMS-0847524 and DMS-0804095.}
\address{School of Mathematics\\
Peking University, Beijing, China}
\address{Department of Mathematics\\
Princeton University, Princeton NJ 08544}
\email{tian@math.princeton.edu}

\author{Zhenlei Zhang}
\thanks{The second author is supported by National Science Foundation grant of China 09221010056}
\address{School of Mathematics\\
Capital Normal University, Beijing, 100048, China}
\address{Department of Mathematics\\
Princeton University, Princeton NJ 08544}
\email{zhleigo@yahoo.com.cn}

\begin{document}

\title{Degeneration of K\"ahler-Ricci solitons}

\begin{abstract}
Let $(Y, d)$ be a Gromov¨CHausdorff limit of $n$-dimensional closed
shrinking K\"ahler-Ricci solitons with uniformly bounded volumes and
Futaki invariants. We prove that off a closed subset of codimension
at least 4, Y is a smooth manifold satisfying a shrinking
K\"ahler-Ricci soliton equation. A similar convergence result for
K\"ahler-Ricci flow of positive first Chern class is also obtained.
\end{abstract}

\maketitle

\section{Introduction}

The degeneration of manifolds with bounded Ricci curvature has been
extensively studied, cf. \cite{ChCo1, ChCo2, ChCo3, ChCoTi, Ch1,
ChTi}, or see \cite{Ch2} for a survey of related results.

Let $(Y,d)$ be a Gromov-Hausdorff limit of a sequence of manifolds
$(M_i,g_i)$ whose Ricci curvature is bounded uniformly from below.
In the noncollapsing case, in \cite{ChCo1} the authors proved that
the singular set of $Y$ is of codimension at least $2$; if the Ricci
curvature admits a two-sided uniform bound, the singular set is
closed and, combining with the result of Anderson \cite{An2}, the
regular set is a $C^{1,\alpha}$ Riemannian manifold. The cases when
$M_i$ has special holonomy or $L^p$ bounded curvature were
considered in \cite{ChCoTi}, \cite{Ch1} and \cite{ChTi}. In
\cite{ChCo1}-\cite{ChCo3}, same structure theorems were proved even
for the collapsing case.

When $(M_i,g_i)$ is Einstein, the metric is smooth and Einstein on
the regular part of $Y$ and the convergence takes place smoothly
there. For shrinking Ricci solitons, under certain curvature
conditions, the degeneration property has been studied in
\cite{CaSe, ZhX, We, WaCh, Zh1}. These results gave generalizations
of orbifold compactness theorem of Einstein manifolds \cite{An1,
BaKaNa, Ti1}. Recently, in \cite{Zh2}, the second-named author
observed that the curvature condition is somehow unnecessary: if
$M_i$ are shrinking Ricci solitons without any curvature assumption
in a prior, then $Y$ has closed singular set whose codimension is at
least $2$ and the regular part is a smooth manifold satisfying the
shrinking Ricci soliton equation.

In this short note, we will consider the shrinking K\"ahler-Ricci
solitons. Here a shrinking K\"ahler-Ricci soliton means a K\"ahler
manifold $(M,g)$ which satisfies
\begin{equation}\label{e1}
\left\{ \begin{array}{ll}
R_{i\bar{j}}+\nabla_i\nabla_{\bar{j}}u=g_{i\bar{j}},\\
\nabla_i\nabla_ju=\nabla_{\bar{i}}\nabla_{\bar{j}}u=0,
\end{array} \right.
\end{equation}
for some smooth function $u$. Associated to the shrinking
K\"ahler-Ricci soliton, the Futaki invariant, evaluated at $\nabla
u$, is given by $\int_M|\nabla u|^2dv$.

We are going to prove the following theorem.

\begin{theorem}\label{t1}
Let $(M_i,g_i)$ be a sequence of $n$-dimensional compact shrinking
K\"ahler-Ricci solitons such that $c_1(M_i)^n\leq C$ and the Futaki
invariant $\leq C$ for a uniform constant $C$. Then, by taking a
subsequence if necessary,
$(M_i,g_i)\stackrel{d_{GH}}{\longrightarrow}(Y,d)$, where $(Y,d)$ is
a path metric space with a closed singular set $\mathcal{S}$ of
codimension at least 4. On the regular set $Y\backslash\mathcal{S}$,
$d$ is induced by a smooth K\"ahler metric which satisfies a
K\"ahler-Ricci soliton equation. Furthermore, the convergence takes
place smoothly on $Y\backslash\mathcal{S}$.
\end{theorem}

We give some remarks about our theorem.

\begin{remark}
The hypothesis in the theorem implies a uniform upper bound on the
diameter of $M_i$. In general, the diameter is difficult to control;
while on the other hand, the Futaki invariant is much easier to
handle, since it is defined on the finite dimensional vector space
generated by holomorphic vector fields. We also remark that the
upper bound on Futaki invariant can be replaced by an upper bound of
$F_{-\nabla u_i}(-\nabla u_i)=\int_{M_i}|\nabla u_i|^2e^{u_i}dv$,
the modified Futaki invariant defined in \cite{TiZh}.
\end{remark}

\begin{remark}
Due to a theorem from algebraic geometry, given $n$, there exist
only finitely many families of $M$ with $c_1(M)>0$. This should
imply that there is an upper bound on both $c_1(M)^n$ and the Futaki
invariant. Hence, the conditions in the theorem hold automatically.
\end{remark}

\begin{remark}
If the curvature of $M_i$ admits a uniform $L^2$ bound, then,
following the proof of Theorem 1.23 in \cite{ChCoTi}, one can also
show that the $(2n-4)$-dimensional Hausdorff measure
$\mathcal{H}^{2n-4}(\mathcal{S})<\infty$.
\end{remark}

In \S 2, we recall and prove some preliminaries about shrinking
Ricci solitons, then in \S 3, we provide a proof of the theorem. In
the last section \S 4, we prove a theorem on the degeneration of a
K\"ahler-Ricci flow $(M,g(t))$ on a compact manifold with positive
first Chern class. In the course of the proof in \S 4, we show the
$L^\infty$ estimate for the minimizer of $\mu(g(t),\frac{1}{2})$.

\section{Preliminaries about shrinking Ricci solitons}

We recall and prove some basic estimates about compact shrinking
K\"ahler-Ricci solitons in this section. Let $(M,g)$ be a compact
shrinking K\"ahler-Ricci soliton with potential function $u$. We
suppose throughout this note that $u$ is normalized such that
\begin{equation}\label{e23}
\int_M(2\pi)^{n}e^{-u}dv=1.
\end{equation}

\subsection{Some basic results}

It's well-known that the following identities
\begin{eqnarray}
&R+\triangle u=n,\label{e24}\\
&R+|\nabla u|^2=u-\bar{u}+n\label{e25}
\end{eqnarray}
hold, where $R$ denotes the scalar curvature of $g$ and
$\bar{u}=\int_Mu(2\pi)^{-n}e^{-u}dv$. Obviously by definition $\inf
u\leq\bar{u}$; while on the other hand, by a result of Ivey
\cite{Iv}, $R>0$, thus equation (\ref{e25}) implies that $\inf
u>\bar{u}-n$. Thus,
\begin{equation}\label{e26}
\bar{u}-n<\inf u\leq\bar{u},
\end{equation}
\begin{equation}\label{e26.5}
 \int_M(-\triangle u+|\nabla u|^2)e^{-u}dv=\int_M(u-\bar{u})e^{-u}dv=0.
\end{equation}

\subsection{Perelman's entropy functional}

Recall that Perelman's entropy functional for a closed K\"ahler manifold $(M,g)$ is defined by \cite{Pe}
\begin{equation}\label{e27}
\mathcal{W}(g,f,\tau)=\int_M\big(2\tau(R+|\nabla f|^2)+f-2n\big)(4\pi\tau)^{-n}e^{-f}dv,
\end{equation}
where $f\in C^\infty$ and $\tau>0$ is any constant. Then define the $\mu$ functional via
\begin{equation}\label{e28}
\mu(g,\tau)=\inf_f\{\mathcal{W}(g,f,\tau)|\int_M(4\pi\tau)^{-n}e^{-f}dv=1\}.
\end{equation}
We remark that, according to Perelman's monotonicity theorem about
$\mathcal{W}$ along the Ricci flow, if $(M,g)$ is a shrinking
K\"ahler-Ricci soliton with potential function $u$, then
\begin{equation}\label{e29}
\mu(g,\frac{1}{2})=\int_M(R+|\nabla u|^2+u-2n)(2\pi)^{-n}e^{-u}dv.
\end{equation}
Applying formulas (\ref{e24}) and (\ref{e26.5}), we can rewrite (\ref{e29}) as
\begin{equation}\label{e210}
\mu(g,\frac{1}{2})=\bar{u}-n.
\end{equation}

The entropy relates the local collapsing information of a shrinking
Ricci soliton. More precisely, by an argument as in \cite{SeTi}, we
have the following lemma.

\begin{lemma}\label{l21}
There exist a function $D=D(A,V,n)$ for any $A,V>0$ and integer $n$
satisfying the following. Let $(M,g)$ be an $n$-dimensional
shrinking K\"ahler-Ricci soliton such that
$\mu(g,\frac{1}{2})\geq-A$ and $\Vol(M)\leq V$, then its diameter
$\diam(M)\leq D$.
\end{lemma}

\subsection{Bound $\mu$ in terms of Futaki invariant}

When $(M,g)$ is a shrinking K\"ahler-Ricci soliton, the Futaki invariant, evaluated at the gradient vector field of $u$, is given by
\begin{equation}\label{e211}
\mathrm{F}=\mathrm{F}(\nabla u)=\int_M|\nabla u|^2dv.
\end{equation}

We will prove the following lemma.

\begin{lemma}\label{l22}
There exists a positive constant $c=c(n)$ such that for any compact shrinking K\"ahler-Ricci soliton $(M,g)$, we have
\begin{equation}\label{e212}
\mu(g,\frac{1}{2})\geq-c(1+\mathrm{F}).
\end{equation}
\end{lemma}
\begin{proof}
Noting that the volume of a compact shrinking K\"ahler-Ricci soliton
always admits a lower bound by a constant depending only on $n$, so
in view of relation (\ref{e210}), it suffices to show that
$$\bar{u}\geq\min\{-4n,2\ln\Vol(M)-4n\ln(2\pi),-\frac{8\mathrm{F}}{\Vol(M)}\}.$$

Assume that $\bar{u}\leq-4n$. First of all, we claim that
$$\Vol\{x\in M|u(x)\leq\frac{\bar{u}}{2}\}\leq(2\pi)^ne^{\frac{\bar{u}}{2}}.$$
Actually, this follows from the normalizing condition (\ref{e23}):
$$1=\int_M(2\pi)^{-n}e^{-u}dv\geq\int_{\{u\leq\frac{\bar{u}}{2}\}}(2\pi)^{-n}e^{-u}dv\geq(2\pi)^{-n}e^{-\frac{\bar{u}}{2}}\Vol\{u\leq\frac{\bar{u}}{2}\}.$$

Suppose that $\bar{u}\leq 2\ln\Vol(M)-4n\ln(2\pi)$, then
$\bar{u}\leq 2\ln\frac{\Vol(M)}{2}-2n\ln(2\pi)$. We have
$\Vol(M)\geq 2(2\pi)^ne^{\frac{\bar{u}}{2}}\geq
2\Vol\{u\leq\frac{\bar{u}}{2}\}$, which implies that
$$\Vol\{u\geq\frac{\bar{u}}{2}\}=\Vol\big(M\backslash\{u<\frac{\bar{u}}{2}\}\big)\geq\frac{1}{2}\Vol(M).$$
Integrating equation (\ref{e25}) we get
\begin{eqnarray}
\int_M|\nabla u|^2dv&=&\int_M(u-\bar{u}+n-R)dv\nonumber\\
&=&\int_M(u-\bar{u}+\triangle u)dv\nonumber\\
&=&\int_M(u-\bar{u})dv\nonumber\\
&=&\int_{\{u\geq\frac{\bar{u}}{2}\}}(u-\bar{u})dv+\int_{\{u\leq\frac{\bar{u}}{2}\}}(u-\bar{u})dv\nonumber\\
&\geq&-\frac{\bar{u}}{2}\Vol\{u\geq\frac{\bar{u}}{2}\}-n\Vol\{u\leq\frac{\bar{u}}{2}\}\nonumber\\
&\geq&(-\frac{\bar{u}}{4}-\frac{n}{2})\Vol(M)\nonumber\\
&\geq&-\frac{\bar{u}}{8}\Vol(M),\nonumber
\end{eqnarray}
where we used (\ref{e26}) in the second inequality and the assumption $\bar{u}\leq-4n$ in the last inequality. This finishes the proof of the lemma.
\end{proof}

\subsection{Bound the Ricci curvature up to a conformal change}

Let $(M,g)$ be a K\"ahler-Ricci soliton whose diameter is less than $D$. By an easy computation, we have the following bound,
\begin{equation}\label{C1 bound}
\|u\|_{C^0}+\|\nabla u\|_{C^0}+\|R\|_{C^0}\leq C_1
\end{equation}
for some $C_1=C_1(n,D)$, which does not depend on the specified
K\"ahler-Ricci soliton; see \cite[\S 3.1]{Zh2} or \cite{SeTi} for a
proof. As a consequence, the Ricci curvature of
$\tilde{g}=e^{-\frac{u}{n-1}}g$, which is a conformal Hermitian
metric of $g$, admits a two-sided bound \cite[\S 3.2]{Zh2}
\begin{equation}\label{Ricci bound}
|Ric(\tilde{g})|_{\tilde{g}}\leq C_2
\end{equation}
by a constant $C_2=C_2(n,D)$. The other observation is that $\tilde{g}$ and $g$ itself are uniformly equivalent to each other.

\section{Proof of Theorem \ref{t1}}

Now let $(M_i,g_i)$ be a sequence of shrinking K\"ahler-Ricci
solitons which satisfies a uniform upper bound on both $c_1(M_i)^n$
and the Futaki invariant, then by Lemma \ref{l21} and \ref{l22},
there exists $D$ independent of $i$ such that $\diam(M_i)\leq D$. By
\cite{Zh2}, there exists a compact path metric space $(Y,d_\infty)$
such that
\begin{equation}\label{convergence0}
(M_i,g_i)\stackrel{d_{GH}}{\longrightarrow}(Y,d_\infty)
\end{equation}
along a subsequence. Denote by $\mathcal{R}$ the smooth part of $Y$
and $g_\infty$ the smooth metric on $\mathcal{R}$ which induces
$d_\infty$. From \cite{Zh2}, the convergence takes place smoothly on
$\mathcal{R}$. Thus, $\mathcal{R}$ admits a natural complex
structure for which the metric $g_\infty$ is K\"ahlerian and
satisfies a shrinking K\"ahler-Ricci soliton equation on
$\mathcal{R}$.

Following Cheeger and Colding \cite{ChCo1}, denote by
$\mathcal{S}_k$ the set of points $y\in Y$ any of whose tangent cone
splits off at most a $k$-dimensional Euclidean space. By
\cite{ChCo1}, together with Claim 3.9 in \cite{Zh2}, we have
$\dim(\mathcal{S}_k)\leq k$, $\mathcal{S}_k\subset\mathcal{S}_{k+1}$
for all $k$. Furthermore, according to \cite[Theorem 1.1]{Zh2}, the
singular set equals $\mathcal{S}_{2n-2}$. So, to prove the Theorem
\ref{t1}, it suffices to show that both
$\mathcal{S}_{2n-2}\backslash\mathcal{S}_{2n-3}$ and
$\mathcal{S}_{2n-3}\backslash\mathcal{S}_{2n-4}$ are empty sets.

\subsection{The tangent cone of $Y$}

From now on, we fix a singular point $y\in\mathcal{S}=Y\backslash\mathcal{R}$ and a tangent cone at $y$, which is given by
\begin{equation}\label{tangent cone}
(Y_y,d_y,o)=\lim_{j\rightarrow\infty}(Y,\rho_j^{-1}d_\infty,y)
\end{equation}
where $\rho_j\rightarrow 0$ is a sequence of positive numbers and
the convergence is taken  in the pointed Gromov-Hausdorff topology.
Then $Y_y=\mathbb{R}^{k}\times C(X)$ for a path metric space $X$
with $\diam(X)\leq 2\pi$, where $\mathbb{R}^{k}$ denotes the
$k$-Euclidean space. We assume that, in the splitting, $k$ is
maximal. Denote by $\mathcal{R}_y$ the regular part of $Y_y$.

Let $y_i\in M_i$ such that $y_i\rightarrow y$. Combining with (\ref{convergence0}), along a subsequence $i_j\rightarrow\infty$, we have
\begin{equation}\label{convergence1}
(M_{i_j},\rho_j^{-2}g_{i_j},y_{i_j})\longrightarrow(Y_y,d_y,o),
\end{equation}
as $j\rightarrow\infty$. By the argument in \cite{Zh2}, using that
$\rho_j^{-2}g_{i_j}$ is a shrinking Ricci soliton, one can check
that the singular set $Y_y\backslash\mathcal{R}_y$ is closed and has
Hausdorff dimension $\leq 2n-2$. Denote by $g_y$ the metric on
$\mathcal{R}_y$ which induces $d_y$.

\begin{claim}\label{c31}
Passing a subsequence if necessary, the convergence in (\ref{convergence1}) is smooth on $\mathcal{R}_y$.
\end{claim}
\begin{proof}
Basically, the proof is the same as the proof of Theorem 1.1 in
\cite{Zh2}. We sketch it here; the main step is to show the
$C^\alpha$ convergence on $\mathcal{R}_y$.

Let $u_i$ be the associated potential function of $g_i$. Then,
passing a subsequence if necessary, $u_i$ converges to a Lipschitz
function $u_\infty$ on $Y$. Denote by
$\tilde{g}_i=e^{\frac{1}{n-1}(u_i(y_i)-u_i)}g_i$ the conformal
Hermitian metric on $M_i$. The important feature is that
$\tilde{g}_i$ has a two-sided bound on Ricci curvature; see
(\ref{Ricci bound}). Thus,
\begin{equation}\label{e31}
|Ric(\rho_j^{-2}\tilde{g}_{i_j})|\rightarrow 0 \hspace{0.3cm}\mbox{as }j\rightarrow\infty.
\end{equation}
Now applying the uniform bound of $u_i$, cf. (\ref{C1 bound}), one verifies that (passing a subsequence if necessary)
\begin{equation}\label{convergence2}
(M_{i_j},\rho_j^{-2}\tilde{g}_{i_j},y_{i_j})\longrightarrow(Y_y,d_y,o)
\end{equation}
as $j\rightarrow\infty$; see the proof of Claim 3.9 in \cite{Zh2}
for details. By \cite{An2}, the later convergence is in the
$C^{1,\alpha}$ topology on $\mathcal{R}_y$. More precisely, for any
compact subset $K\subset\mathcal{R}_y$, there exist a sequence of
diffeomorphic embeddings $\phi_j:K\rightarrow M_{i_j}$ such that
\begin{eqnarray}\label{C1alpha convergence}
\|\phi_j^*(\rho_j^{-2}\tilde{g}_{i_j})-g_y\|_{C^{1,\alpha}(K)}\rightarrow 0
\end{eqnarray}
as $j\rightarrow\infty$. With respect to the metric
$\rho_j^{-2}\tilde{g}_{i_j}$ on $\phi_j(K)$, by the $C^1$ estimate
of the potential function (\ref{C1 bound}), we have
\begin{equation}\label{e32}
\|u_{i_j}(y_{i_j})-u_{i_j}\|_{C^0(\phi_j(K))}+\|\nabla u_{i_j}\|_{C^0(M_{i_j})}\rightarrow 0,\hspace{0.3cm}\mbox{ as }j\rightarrow\infty.
\end{equation}
Passing to $g_{i_j}$, the pull-backed metric
\begin{equation}\label{e33}
\phi_j^*(\rho_j^{-2}g_{i_j})=\phi_j^*(e^{-\frac{1}{n-1}(u_{i_j}(y_{i_j})-u_{i_j})}\rho_j^{-2}\tilde{g}_{i_j})=
e^{-\frac{1}{n-1}(u_{i_j}(y_{i_j})-u_{i_j}\circ\phi_j)}\phi_j^*(\rho_j^{-2}\tilde{g}_{i_j})
\end{equation}
has a uniform $C^\alpha$ bound on $K$. Hence, combining this with
(\ref{C1alpha convergence}) and (\ref{e32}) we get the $C^\alpha$
convergence of $\phi_j^*(\rho_j^{-2}g_{i_j})$  to $g_y$ on $K$.

Being known the $C^\alpha$ convergence, the $C^\infty$ convergence
of (\ref{convergence1}) (along a subsequence) follows from a
regularity argument based on Perelman's pseudolocality theorem
\cite{Pe} (see also Theorem 4.2 in \cite{FZZ}) and Shi's gradient
estimate \cite{Shi, Ha95b} for Ricci flow, since the metrics in
consideration are shrinking Ricci solitons and can be seen as time
slices of the corresponding Ricci flows. See the proof of Claim 3.10
in \cite{Zh2}, for example.  We omit the details here.
\end{proof}

Let $J_{i_j}$ be the canonical complex structure on $M_{i_j}$. By
the smooth convergence, passing a subsequence if necessary, there
exist a family of compact subsets $K_j$ satisfying $K_j\subset
K_{j+1},\cup K_j=\mathcal{R}$, and embeddings
$$\phi_j:K_j\longrightarrow M_{i_j}$$
such that
$\phi_j^*(\rho_j^{-2}g_{i_j})\stackrel{C^\infty}{\longrightarrow}g_y$
and $\phi_j^*J_{i_j}\stackrel{C^\infty}{\longrightarrow}J_\infty$
for a complex structure $J_\infty$ on $\mathcal{R}_y$. Obviously
$J_\infty$ is parallel with respect to $g_y$. We claim that
$J_\infty$ splits as a product complex structure on the regular part
$\mathcal{R}_y$, which implies that, in the splitting of $Y_y$, $k$
is even.

We need the following assertion.

\begin{claim}\label{c32}
Let $\nabla$ and $\tilde{\nabla}$ be the Levi-Civita connections of
$\rho_j^{-2}g_{i_j}$ and
$\rho_j^{-2}\tilde{g}_{i_j}=e^{\frac{1}{n-1}(u_{i_j}(y_{i_j})-u_{i_j})}\rho_j^{-2}
g_{i_j}$ respectively. Then we have
\begin{equation}\label{e34}
\|\tilde{\nabla}-\nabla\|_{C^0(M_{i_j})}\rightarrow 0,\hspace{0.3cm}\mbox{ as }j\rightarrow\infty,
\end{equation}
where the $C^0$ norm is taken with respect to the metric $\rho_j^{-2}\tilde{g}_{i_j}$.
\end{claim}
\begin{proof}
Let $\tilde{\Gamma}$ and $\Gamma$ be the Christoffel symbols of the
Levi-Civita connections of $\rho_j^{-2}\tilde{g}_{i_j}$ and
$\rho_j^{-2}g_{i_j}$. Then, in local coordinate, the connections
have the difference
\begin{eqnarray}
\tilde{\nabla}-\nabla&=&(\tilde{\Gamma}_{pq}^s-\Gamma_{pq}^s)dx^p\otimes dx^q\otimes\frac{\partial}{\partial x^s}\nonumber\\
&=&\frac{1}{2(n-1)}(\frac{\partial u_{i_j}}{\partial x^p}\delta_{qs}-\frac{\partial u_{i_j}}{\partial x^q}\delta_{ps}+\frac{\partial u_{i_j}}{\partial x^t}g^{st}g_{pq})
dx^p\otimes dx^q\otimes\frac{\partial}{\partial x^s}\nonumber.
\end{eqnarray}
From the $C^1$ estimate of $u_{i_j}$, see (\ref{C1 bound}), we have
$$\|\nabla u_{i_j}\|_{C^0(M_{i_j})}\rightarrow 0,\hspace{0.3cm}\mbox{ as }j\rightarrow\infty,$$
where the $C^0$ norm is taken with respect to $\rho_j^{-2}\tilde{g}_{i_j}$. The desired result then follows.
\end{proof}

If $v$ is an almost parallel vector field of $M_{i_j}$ with respect
to $\rho_j^{-2}\tilde{g}_{i_j}$, then $J(v)$ is also an almost
parallel vector field with respect to $\rho_j^{-2}\tilde{g}_{i_j}$.
This is because
$$\tilde{\nabla}(J(v))=(\tilde{\nabla}J)v+J(\tilde{\nabla}v)=\big((\tilde{\nabla}-\nabla)J\big)v+J(\tilde{\nabla}v),$$
where $\tilde{\nabla}-\nabla\rightarrow 0$, while $J$ is an almost
isometric endomorphism of $TM$ with metric
$\rho_j^{-2}\tilde{g}_{i_j}$ whenever $j$ is large enough. Then, in
view of the convergence (\ref{convergence2}), by the explanation in
\cite[Page 399]{ChTi}, if $v$ is the gradient of a harmonic function
which induces an almost splitting, so does $Jv$. For more details,
see the proof of Theorem 9.1 in \cite{ChCoTi}; note that the vector
field $v$ satisfies assumption of Lemma 9.14 in \cite{ChCoTi} if and
only if $Jv$ does, up to a modification of the constants $c$ and
$\delta$ there. As a direct consequence, passing to the limit space
$Y_y$, we have that
\begin{equation}\label{e35}
\mathcal{S}_{2l+1}\backslash\mathcal{S}_{2l}=\emptyset,\hspace{0.3cm}\mbox{ for all integer }l.
\end{equation}
In particular,
\begin{equation}\label{e36}
\mathcal{S}_{2n-3}\backslash\mathcal{S}_{2n-4}=\emptyset.
\end{equation}
Furthermore, the Euclidean factor
$\mathbb{R}^{k}=\mathbb{C}^{\frac{k}{2}}$ with respect to $J_\infty$
and on the regular part of $Y_y$, we have the splitting structure
$$\mathcal{R}_y=\mathbb{C}^{\frac{k}{2}}\times\mathcal{R}(C(X)),$$
where $\mathcal{R}(C(X))$ denotes the regular part of $C(X)$.

In the next subsection, we will make use of the convergences
(\ref{convergence1}) and (\ref{convergence2}) to remove the
singularities of type
$\mathcal{S}_{2n-2}\backslash\mathcal{S}_{2n-3}$ when $M_i$ are
K\"ahlerian.

\subsection{The slice argument and regularity}

As in \cite{ChCoTi, Ch1}, we will show the following regularity
property, which is sufficient to prove the main theorem in view of
the convergence (\ref{convergence1}).

Let $y\in\mathcal{S}_{2n-2}\backslash\mathcal{S}_{2n-3}$ be a
typical singular point. Then the tangent cone
$Y_y=\mathbb{R}^{2n-2}\times C(S_t)$ where $S_t$ is a round circle
with circumference $t<2\pi$. Let $0\in\mathbb{R}^{2n}$,
$\underline{0}\in\mathbb{R}^{2n-2}$ be the origins and $x^*\in
C(S_t)$ be the vertex. Then $o=(\underline{0},x^*)$. A typical point
of $Y_y$ can be expressed as $(z,r,x)$ where
$z\in\mathbb{R}^{2n-2}$, $x\in S_t$ and $r\geq0$ denotes the radial
coordinate on $C(S_t)$.

\begin{proposition}\label{p31}
For all $\eta>0$, there exists $j_0$ such that
\begin{equation}\label{e37}
d_{GH}\big(B_{\rho_j^{-2}g_{i_j}}(y_{i_j},1),B(0,1)\big)\leq\eta,\hspace{0.3cm}\forall j\geq j_0.
\end{equation}
Here, $B(0,1)$ denotes the unit ball in the $2n$-Euclidean space.
\end{proposition}

Let $\Psi=\Psi(j,l|\delta,D,n)$ be a positive constant that may vary in the following argument such that for any fixed $n,\delta$ and $D$,
$$\Psi(j,l|\delta,D,n)\rightarrow0,\hspace{0.3cm}\mbox{ as }j\rightarrow\infty\mbox{ and }l\rightarrow\infty.$$

\begin{proof}[Proof of Proposition \ref{p31}]
For simplicity, we fix a specified manifold
$(M,g)=(M_{i_j},\rho_j^{-2}g_{i_j})$ and let
$$\tilde{g}=e^{\frac{1}{n-1}(u_{i_j}(y_{i_j})-u_{i_j})}\rho_j^{-2}g_{i_j},\hspace{0.3cm}\tilde{\tilde{g}}=e^{\frac{1}{n}(u_{i_j}(y_{i_j})-u_{i_j})}
\rho_j^{-2}g_{i_j}.$$ Let $u=u_{i_j}$ be the potential function and
$\rho=\rho_j$ be the rescaling factor. The metric
$\tilde{\tilde{g}}$ is introduced for the convenience of computing
the Chern class.

Since (\ref{convergence2}) holds, for any $l$, we have
\begin{equation}\label{e38}
d_{GH}(B_{\tilde{g}}(y,l),B(0,l))<l^{-1}
\end{equation}
whenever $j$ is large enough. Let $F:B_{\tilde{g}}(y,l)\rightarrow
Y_y$ be a Gromov-Hausdorff approximation realizing the convergence
(\ref{convergence2}). We remark that by the proof of Claim
\ref{c31}, $F$ also realizes the convergence (\ref{convergence1}).
So by the smooth convergence on
$\mathbb{R}^{2n-2}\times(C(S_t)\backslash\{x^*\})$, we may assume
that $F$ is holomorphic and diffeomorphic on the range outside of
$\mathbb{R}^{2n-2}\times B(x^*,\frac{1}{3})$. Recall that by
\cite{ChCoTi}, see also \cite{Ch1}, there exist smooth functions
\begin{eqnarray}
\Lambda=(\Phi,{\rm r}):B_{\tilde{g}}(y,3)\rightarrow\mathbb{R}^{2n-2}\times\mathbb{R}_+\label{e39}
\end{eqnarray}
such that
\begin{equation}\label{e312}
\|\Phi-z\circ F\|_{C^0}+\|{\rm r}-r\circ F\|_{C^0}\leq\Psi.
\end{equation}
Furthermore, if we set
\begin{equation}\label{e314}
\Sigma_{z,r}=\Phi^{-1}(z)\cap\Lambda^{-1}(B(\underline{0},1)\times[0,r]),
\end{equation}
then, there exist subsets $B_l\subset B(\underline{0},1)$ and
$D_l\subset B(\underline{0},1)\times[0,1]$ consisting of regular
values of $\Phi$ and $\Lambda$ such that
\begin{eqnarray}
&\Vol(B_l)\geq(1-\Psi)\Vol(B(\underline{0},1)),\label{e315}\\
&\Vol(D_l)\geq(1-\Psi)\Vol(B(\underline{0},1)\times[0,1]),\label{e316}
\end{eqnarray}
and that for all $z\in B_l$ and $(z,r)\in D_l$,
\begin{eqnarray}
&|\Vol_{\tilde{g}}(\Sigma_{z,r})-\frac{r^2}{2}t|\leq\Psi.\label{e317}\\
&|\Vol_{\tilde{g}}(\Lambda^{-1}(z,r))-rt|\leq\Psi.\label{e318}
\end{eqnarray}
Note that if $(z,r)\in D_l$, then $\Lambda^{-1}(z,r)=\partial\Sigma_{z,r}$.

Then one can use the differential character of the first Chern
class on suitable slices of $\Lambda^{-1}(z,r)$ to give a proof of
the proposition; see \cite{ChTi}, \cite{Ch2} and \cite{Ch1} for
details. The idea to use the differential character first appeared
in \cite{ChCoTi}. In the following, we give an alternative and
direct proof of the proposition based on the transgression theory.
The argument is also inspired by the work of Cheeger, Colding and
Tian for K\"ahler manifolds with bounded Ricci curvatures, cf.
\cite{ChCoTi}, \cite{ChTi}, \cite{Ch2} and \cite{Ch1}.

Choose $(z,r)\in D_l$ such that $\frac{1}{2}\leq r\leq 1$ and let
$L$ be the determinant bundle of $TM$ restricted to $\Sigma_{z,r}$.
Denote also by $\tilde{\tilde{g}}$ the induced Hermitian metric on
$L$. Denote by $\widetilde{\widetilde{\Theta}}$ the curvature form
associated to the Chern connection of $\tilde{\tilde{g}}$. Let
$S\Sigma_{z,r}$ be the unit circle bundle and
$\pi:S\Sigma_{z,r}\rightarrow\Sigma_{z,r}$ be the projection.  Then
there exist connection form $\omega$ and curvature form $\Omega$ on
$S\Sigma_{z,r}$ such that
\begin{equation}
\Omega=\pi^*\widetilde{\widetilde{\Theta}}=d\omega.
\end{equation}

Note that the curvature form $\widetilde{\widetilde{\Theta}}$ gives the first Chern class of $L$:
\begin{eqnarray}
\frac{\sqrt{-1}}{2\pi}\widetilde{\widetilde{\Theta}}&=&P_{c_1}(L)|_{\Sigma_{z,r}}\nonumber\\
&=&-\frac{\sqrt{-1}}{2\pi}\partial_p\partial_{\bar{q}}\log\det(\tilde{\tilde{g}})dz^p\wedge d\bar{z}^q|_{\Sigma_{z,r}}\nonumber\\
&=&\frac{\sqrt{-1}}{2\pi}(R_{p\bar{q}}+\partial_p\partial_{\bar{q}}u)dz^p\wedge d\bar{z}^q|_{\Sigma_{z,r}}\nonumber\\
&=&\frac{\sqrt{-1}}{2\pi}\rho^2 g_{p\bar{q}}dz^p\wedge d\bar{z}^q|_{\Sigma_{z,r}}\nonumber,
\end{eqnarray}
whose norm tends to $0$ uniformly as $j\rightarrow\infty$. Now
choose a transversal section $s:\Sigma_{z,r}\rightarrow L$ with
finite zero set $\{p_k\}_{k=1}^N$ such that $s=x\circ F$ on
$\partial\Sigma_{z,r}=\Lambda^{-1}(z,r)$ with respect to local
trivialization by $F$ on a neighborhood of $\Lambda^{-1}(z,r)$ .
Then we have
$(\frac{s}{|s|})^*\Omega=\widetilde{\widetilde{\Theta}}$ outside of
the zero set, so by (\ref{e317}) and using the $C^1$ equivalence of
$\tilde{g}$ and $\tilde{\tilde{g}}$,
\begin{eqnarray}
\Psi&=&\frac{\sqrt{-1}}{2\pi}\int_{\Sigma_{z,r}}\widetilde{\widetilde{\Theta}}\nonumber\\
&=&\lim_{\epsilon\rightarrow 0}\frac{\sqrt{-1}}{2\pi}\int_{\Sigma_{z,r}-\cup_{k}D(p_k,\epsilon)}\widetilde{\widetilde{\Theta}}\nonumber\\
&=&\lim_{\epsilon\rightarrow 0}\frac{\sqrt{-1}}{2\pi}\int_{\Sigma_{z,r}-\cup_{k}D(p_k,\epsilon)}(\frac{s}{|s|})^*\Omega\nonumber\\
&=&\lim_{\epsilon\rightarrow 0}\frac{\sqrt{-1}}{2\pi}\int_{\Sigma_{z,r}-\cup_{k}D(p_k,\epsilon)}d(\frac{s}{|s|})^*\omega\nonumber\\
&=&\sum_k\lim_{\epsilon\rightarrow 0}\frac{\sqrt{-1}}{2\pi}\int_{\partial D(p_k,\epsilon)}(\frac{s}{|s|})^*\omega+\frac{\sqrt{-1}}{2\pi}\int_{\Lambda^{-1}(z,r)}(\frac{x}{|x|}\circ F)^*\omega\nonumber\\
&=&\sum_k\lim_{\epsilon\rightarrow 0}\frac{\sqrt{-1}}{2\pi}\int_{\frac{s}{|s|}(\partial D(p_k,\epsilon))}\omega+\frac{\sqrt{-1}}{2\pi}\int_{\frac{x}{|x|}\circ F(\Lambda^{-1}(z,r))}\omega\nonumber.
\end{eqnarray}
Here, $D(p_k,\epsilon)$ denotes an $\epsilon$ ball around $p_k$ in
$\Phi^{-1}(z)$. In the last formula, each limit term
$\lim_{\epsilon\rightarrow
0}\frac{\sqrt{-1}}{2\pi}\int_{\frac{s}{|s|}(\partial
D(p_k,\epsilon))}\omega$ is an integer since locally $\omega$ can be
written as $d\log x+\tilde{\tilde{\theta}}$ where
$\tilde{\tilde{\theta}}$ denotes the horizontal term of $\omega$;
while $\frac{\sqrt{-1}}{2\pi}\int_{\frac{x}{|x|}\circ
F(\Lambda^{-1}(z,r))}\omega$ appoximates $\frac{t}{2\pi}$, the
corresponding integration on the circle $\{(r,x)|x\in S_t\}$ in the
cone $C(S_t)$, as $j\rightarrow\infty$ since the metric
$\tilde{\tilde{g}}$ approaches $g_y$ uniformly in the $C^1$ topology
around $\Lambda^{-1}(z,r)$ when $(z,r)\in D_l$ with $\frac{1}{2}\leq
r\leq 1$. Notice that the connection form $\omega$ is determined by
first derivative of the metric. This is sufficient to show that
$\frac{t}{2\pi}$ is an integer; then applying the relation
(\ref{e318}) we infer that $t=2\pi$.

This completes the proof of the proposition.
\end{proof}

Now, the proof of Theorem \ref{t1} follows easily.

\begin{proof}[Proof of Theorem \ref{t1}]
By (\ref{e36}) and Proposition \ref{p31},
$\mathcal{S}=\mathcal{S}_{2n-4}$. As the convergence on the regular
part is smooth, it satisfies a shrinking K\"ahler-Ricci soliton
equation there. The proof is completed.
\end{proof}

\section{On the K\"ahler-Ricci flow with positive $c_1$}

In this section, we use the above method to study the degeneration
of a K\"ahler-Ricci flow on a K\"ahler manifold $M$ with positive
$c_1(M)$. Let $g(t),t\in[0,\infty),$ be a solution to the
K\"ahler-Ricci flow on $M$
\begin{equation}\label{e40}
\frac{\partial}{\partial t}g_{i\bar{j}}=-R_{i\bar{j}}+g_{i\bar{j}}
\end{equation}
whose K\"ahler form lies in the K\"ahler class $\pi c_1(M)$. It is
well-known that the K\"ahler-Ricci flow preserves the volume, say
$\V=\Vol_{g(0)}(M)$. It is also known that $g(t)$ has the same
K\"ahler class. Thus, Hodge theory gives a family of potentials
$\phi(t)$ such that $g(t)=g(0)+\partial\bar{\partial}\phi(t)$. The
K\"ahler-Ricci flow is then equivalent to the complex Monge-Amp\`ere
equation for $\phi(t)$
\begin{equation}
\frac{\partial\phi(t)}{\partial t}=\log\frac{\det(g(0)+\partial\bar{\partial}\phi(t))}{\det(g(0))}+\phi(t).
\end{equation}

Let $u(t)$ be associated Ricci potential in the sense that
\begin{equation}\label{e41}
R_{i\bar{j}}(t)+\nabla_{i}\nabla_{\bar{j}}u(t)=g_{i\bar{j}}(t).
\end{equation}
Suppose $u(t)$ is normalized such that
\begin{equation}
\int e^{-u(t)}dv_{g(t)}=(2\pi)^n.
\end{equation}
Perelman proved the following estimates (see \cite{SeTi} for the proof)
\begin{equation}\label{e46}
\diam(M,g(t))+\|u(t)\|_{C^0}+\|\nabla u(t)\|_{C^0}+\|R(g(t))\|_{C^0}\leq C_1
\end{equation}
for some $C_1$ independent of $t$. It is also shown in \cite{SeTi} that $\mu(g(t),\frac{1}{2})$, the function defined in \S 2.2, has a uniform bound
\begin{equation}
|\mu(g(t),\frac{1}{2})|\leq C_1.
\end{equation}

In \cite{Se}, Sesum studied the convergence of K\"ahler-Ricci flow
with bounded Ricci curvature. One essential point in her argument is
that the metric derivative in time is uniformly bounded (in terms of
the Ricci curvature) and thus the metrics under the K\"ahler-Ricci
flow are locally equivalent to each other. Here, what we are
concerning is a K\"ahler-Ricci flow with uniformly bounded
$(2,0)$-part of $\Hess(u(t))$
\begin{equation}\label{e42}
|\nabla_{\frac{\partial}{\partial z^i}}\nabla_{\frac{\partial}{\partial z^j}}u(t)|\leq C_2.
\end{equation}
As we will see, the convergence result is same as the case of bounded Ricci curvature. The argument partially follows Sesum's line.

We first observe that (\ref{e42}) implies an immediate bound of $Ric+\Hess(u)$:
\begin{equation}\label{e43}
|Ric(t)+\Hess(u)(t)|\leq C_3
\end{equation}
for some $C_3=C_3(C_2,n)$. For any $g=g(t)$, let
$\tilde{g}=e^{-\frac{u}{n-1}}g$ be a conformal Hermitian metric and
denote by $\widetilde{Ric}$ its Ricci curvature. Then, cf.
\cite{Be},
\begin{equation}\label{e44}
\widetilde{Ric}=Ric+\Hess(u)+\frac{1}{2n-2}du\otimes du+\frac{1}{2n-2}(\triangle u-|\nabla u|^2)g.
\end{equation}
Together with $R+\triangle u=n$, $\widetilde{Ric}$ can be bounded as follows
\begin{equation}\label{e45}
|\widetilde{Ric}|_{\tilde{g}}\leq C_4e^{\frac{u}{n-1}}(1+|\nabla u|^2+|R|)
\end{equation}
for some $C_4=C_4(C_2,n)$.  Then applying (\ref{e46}) gives the bound
\begin{equation}
|\widetilde{Ric}|_{\tilde{g}}\leq C_5
\end{equation}
for some $C_5$ independent of $t$. Notice that the volume
$\Vol(M,g)$ is constant under the K\"ahler-Ricci flow. By
\cite{Zh2}, for any sequence of times $t_k\rightarrow\infty$, there
exists a compact path metric space $(Y,d)$ such that
$(M,g(t_k))\stackrel{d_{GH}}{\longrightarrow}(Y,d)$. The space $Y$
has a singular set $\mathcal{S}$ of codimension at least 2.
Furthermore, on the regular part
$\mathcal{R}=Y\backslash\mathcal{S}$, $d$ is induced by a
$C^{\alpha}$ metric $g_\infty$ and the convergence takes place
$C^\alpha$ there.

As we are on a K\"ahler-Ricci flow, the metric $g_\infty$ should have more regular property. Indeed we can prove

\begin{theorem}\label{t40}
Let $g(t),t\in[0,\infty),$ be a solution to the K\"ahler-Ricci flow
on a K\"ahler manifold $M$ with positive $c_1(M)$. Let $u(t)$ be a
family of Ricci potentials of $g(t)$. If (\ref{e42}) holds for some
$C_2$ independent of $t$, then along a subsequence
$(M,g(t))\stackrel{d_{GH}}{\longrightarrow}(Y,d)$, where $(Y,d)$ is
a path metric space with a closed singular set $\mathcal{S}$ of
codimension at least 4. On the regular set $Y\backslash\mathcal{S}$,
$d$ is induced by a smooth K\"ahler metric which satisfies a
K\"ahler-Ricci soliton equation. Furthermore, the convergence takes
place smoothly on $Y\backslash\mathcal{S}$.
\end{theorem}

Inspired by the work of \cite{Se}, to prove the higher regularity of
$g_\infty$ on $\mathcal{R}$, we need to consider the Ricci flow
$g_k(t)=g(t_k+t)$. For simplicity, we consider the sequence of flows
$h_k(t)=\phi_{k,t}^*g_k(t)$, $t\in[-t_k,\infty)$, instead of
$g_k(t)$ where $\phi_{k,t}$ is a family of diffeomorphisms of $M$
generated by the (real) gradient vector field $-\frac{1}{2}\nabla
u(t_k+t)$, with $\phi_{k,0}=id_M$. By an easy computation, for each
$k$, $h_k(t)$ satisfies the evolution
\begin{equation}\label{e47}
\frac{\partial}{\partial t}h_k(t)=\phi_{k,t}^*(-Ric(g_k(t))-\Hess(u_k(t)))+h_k(t),\hspace{0.3cm}t\in[-t_k,\infty),
\end{equation}
where $u_k(t)=u(t_k+t)$. By (\ref{e43}) and boundedness of $|\nabla
u(t)|$, one gets the uniform bound of derivative of $h_k(t)$:
\begin{equation}\label{e48}
|\frac{\partial}{\partial t}h_k(t)|\leq C_6
\end{equation}
for some $C_6$ independent of $k$ and $t$.

The essential step is to prove the following pseudolocality theorem.

\begin{theorem}\label{t41}
Assume as above. There exist universal constants
$\delta_0,\epsilon_0>0$ independent of $k$ with the following
property. Let $(\bar{x}_0,\bar{t}_0)\in M\times[-t_k,\infty)$ be a
space-time point such that
\begin{equation}\label{e49}
\Vol_{h_k(t)}(B_{h_k(t)}(x,r))\geq(1-\delta_0)\Vol(B(r))
\end{equation}
for all metric ball $B_{h_k(t)}(x,r)\subset
B_{h_k(t)}(\bar{x}_0,r_0)$ with
$t\in[\bar{t}_0,\bar{t}_0+(\epsilon_0r_0)^2]$, where $B(r)$ denotes
the metric ball of radius $r$ in $2n$-Euclidean space and
$\Vol(B(r))$ denotes its volume, then the Riemannian curvature
tensor satisfies
\begin{equation}\label{e410}
|Rm(h_k(t))|(x)\leq (t-\bar{t}_0)^{-1}
\end{equation}
whenever
\begin{equation}\label{e411}
\dist_{h_k(t)}(\bar{x}_0,x)<\epsilon_0r_0\mbox{ and }\bar{t}_0<t\leq\bar{t}_0+(\epsilon_0r_0)^2.
\end{equation}
In particular,
\begin{equation}\label{e412}
|Rm(h_k(t))|(x_0)\leq (t-\bar{t}_0)^{-1},\hspace{0.3cm}\mbox{ for }\bar{t}_0<t\leq \bar{t}_0+(\epsilon_0r_0)^2.
\end{equation}
\end{theorem}
\begin{proof}
The argument is same as that of Theorem 4.2 in \cite{FZZ}. Rescale
the flow $h_k(t)$ such that $r_0=1$. We may assume that under the
rescaling the metric derivative $\frac{\partial}{\partial t}h_k(t)$
still satisfies the bound (\ref{e48}). This is fullfilled if the
initial $r_0$ is less than 1.

Denote by $\overline{M}_k$ the set of space-time points $(x,t)$
satisfying (\ref{e411}) but the curvature $|Rm(h_k(t))|(x)\geq
(t-\bar{t}_0)^{-1}$. Suppose $\overline{M}_k\neq\emptyset$, then as
did in Claim 1 and Claim 2 of \cite{Pe}, one can choose another
space-time point $(\bar{x}_k,\bar{t}_k)\in\overline{M}$ with
$\bar{t}_k<\epsilon_0^2$ and
$\dist_{h_k(\bar{t}_k)}(\bar{x}_0,\bar{x}_k)\leq\frac{1}{10}$ such
that $|Rm(h_k(t))|(x)\leq 4Q_k$ whenever
\begin{equation}\label{e413}
\bar{t}_k-\frac{1}{2n}Q_k^{-1}\leq t\leq\bar{t}_k,\hspace{0.3cm} \dist_{h_k(\bar{t}_k)}(\bar{x}_k,x)\leq\frac{1}{1000n}(Q_k\epsilon_0^2)^{-1/2},
\end{equation}
where $Q_k=|Rm(h_k(\bar{t}_k))|(\bar{x}_k)$. From the metric
derivative bound (\ref{e48}), the space-time point $(x,t)$ with
(\ref{e413}) satisfies $$\dist_{h_k(t)}(\bar{x}_0,x)\leq
(\frac{1}{10}+\frac{1}{1000n}(Q_k\epsilon_0^2)^{-1/2})e^{C_4\epsilon_0^2}\leq\frac{1}{5}e^{C_4\epsilon_0^2}\leq\frac{1}{2}$$
whenever $\epsilon_0$ is chosen small enough.

Suppose the theorem does not hold, then there exist sequences of
positive numbers $r_k>0$, $\epsilon_k\rightarrow 0$,
$\delta_k\rightarrow 0$ and space-time points
$(\bar{x}_k,\bar{t}_k)$ with $\bar{t}_k>-t_k$ such that (\ref{e49})
is fullfilled but (\ref{e410}) is not for all points in
(\ref{e411}). As above we rescale the flow such that the radius
$r_k=1$. Reconstruct the base space-time point
$(\bar{x}_k,\bar{t}_k)$ as in above process and consider the
sequence of rescaled pointed flow
$$\big(B_{h_k(\bar{t}_k)}(\bar{x}_k,\frac{1}{1000n}(Q_k\epsilon_k^2)^{-1/2}),Q_kh(Q_k^{-1}t+\bar{t}_k),\bar{x}_k\big).$$
One should keep in mind that this flow is really a Ricci flow, up to
diffeomorphic actions. Actually, the family of Riemannian manifolds
$$\big(\phi_{k,\bar{t}_k}^{-1}B_{h_k(\bar{t}_k)}(\bar{x}_k,\frac{1}{1000n}(Q_k\epsilon_k^2)^{-1/2}),
Q_k(\phi^{-1}_{k,Q_k^{-1}t+\bar{t}_k})^*h_k(Q_k^{-1}t+\bar{t}_k),\phi_{k,\bar{t}_k}^{-1}(\bar{x}_k)\big)$$
is really part of the rescaled K\"ahler-Ricci flow
$(M,Q_kg_k(Q_k^{-1}t+\bar{t}_k))$ on the time interval
$t\in[-\frac{1}{2n},0]$. Note that by construction the curvature of
this sequence is uniformly bounded (less than $4$) and the radius of
the ball equals $\frac{\epsilon_k^{-1}}{1000n}$ which tends to
$\infty$ as $k\rightarrow\infty$. Thus, by Hamilton's compactness
theorem for Ricci flow \cite{Ha95a}, this sequence will converge
along a subsequence to another pointed K\"ahler-Ricci flow
$(M_\infty,\bar{g}_\infty(t),x_\infty)$ on the time interval
$(-\frac{1}{2n},0]$. The curvature
$|Rm(\bar{g}_\infty(0))|(x_\infty)=1$; furthermore, by (\ref{e49})
the volume growth has Euclidean lower bound
$$\Vol_{\bar{g}_\infty(0)}(B(x_\infty,r))\geq\Vol(B(r)),\hspace{0.3cm}\forall r>0.$$

Next we will show that $\bar{g}_\infty(0)$ is Ricci flat. Since
$Q_kg_k(Q_k^{-1}t+\bar{t}_k)\rightarrow\bar{g}_\infty(t)$ smoothly
and the scalar curvature of $Q_kg_k(Q_k^{-1}t+\bar{t}_k)$ tends to
zero, the scalar curvature of $\bar{g}_\infty(t)$ vanishes
identically. Thus,
$$\frac{\partial}{\partial t}\bar{g}_\infty(t)=-Ric(\bar{g}_\infty(t)).$$
Under this evolution, the scalar curvature satisfies
$$\frac{\partial}{\partial t}R(\bar{g}_\infty(t))=\triangle R(\bar{g}_\infty(t))+|Ric(\bar{g}_\infty(t))|^2.$$
It follows immediately that $Ric(\bar{g}_\infty(t))\equiv 0$ for $t\in(-\frac{1}{2n},0]$.

Together with Bishop-Gromov volume comparison theorem the volume
growth condition implies that $\bar{g}_\infty(0)$ is actually flat.
This contradicts with $|Rm(\bar{g}_\infty(0))|(x_\infty)=1$. The
contradiction proves the theorem.
\end{proof}

As direct consequences, $g_\infty$ is K\"ahler and, by the argument in \S 3, the singular set $\mathcal{S}$ has codimension at least 4.

It is time to claim the smooth convergence on $\mathcal{R}$. By
$C^\alpha$ convergence, there exist embeddings
$\psi_k:\mathcal{R}\rightarrow M$ such that
$\psi_k^*g_k\stackrel{C^\alpha}{\longrightarrow}g_\infty$ uniformly
on any compact subsets of $\mathcal{R}$.

\begin{claim}\label{c42}
Adjusting the embeddings $\psi_k$ if necessary we have $\psi_k^*g_k\stackrel{C^\infty}{\longrightarrow}g_\infty$ on $\mathcal{R}$.
% and $\psi_k^*u(t_k)\stackrel{C^\infty}{\longrightarrow}u_\infty$
\end{claim}
\begin{proof}
By the uniqueness of Gromov-Hausdorff limit, it suffices to show
that for any
$K_\rho=\{x\in\mathcal{R}|\dist(x,\mathcal{S})\geq\rho\}$ with
$\rho>0$, the metric $g_k$ is $C^\infty$ uniformly bounded on any
$\psi_k(K_\rho)$ whenever $k$ is large enough.

Let $\epsilon_0,\delta_0$ be the constants given in Theorem
\ref{t41}. Fix one $\rho>0$. Then by the $C^{\alpha}$ convergence of
$g_k$ on $\mathcal{R}$, there exists $r_\rho\leq\frac{1}{2}\rho$
such that
\begin{equation}\label{e414}
\Vol_{g_k}(B_{g_k}(x,r))\geq(1-\frac{1}{2}\delta_0)\Vol(B(r)),\hspace{0.3cm}\forall  B_{g_k}(x,r)\subset K_{\rho-r_\rho},
\end{equation}
whenever $k$ is large enough and that
\begin{equation}\label{e415}
e^{-2nC_4(\epsilon_0r_\rho)^2}(1-\frac{1}{2}\delta_0)\geq 1-\delta_0,\hspace{0.3cm}\frac{1}{2}e^{C_4(\epsilon_0r_\rho)^2}\leq 1.
\end{equation}
Fix any $x_0\in\psi_k(K_\rho)$. By the metric derivative estimate (\ref{e48}), we have
\begin{equation}\label{e415.5}
|\frac{d}{dt}\log\dist_{h_k(t)}(p,q)|\leq C_4
\end{equation}
for all $p,q\in M$ and
\begin{equation}\label{e415.75}
|\frac{d}{dt}\log\Vol_{h_k(t)}(U)|\leq\sup|tr_{h(t)}\frac{\partial h(t)}{\partial t}|\leq nC_4
\end{equation}
for any domain $U\subset M$. From these estimates one derives in particular that
\begin{eqnarray}
\Vol_{h_k(t)}(B_{h_k(t)}(x,r))&\geq&\Vol_{h_k(t)}(B_{g_k}(x,e^{-C_4(\bar{t}_k-t)}r))\nonumber\\
&\geq& e^{-nC_4(\epsilon_0r_\rho)^2}\Vol_{g_k}(B_{g_k}(x,e^{-C_4(\bar{t}_k-t)}r))\nonumber\\
&\geq&e^{-nC_4(\epsilon_0r_\rho)^2}(1-\frac{1}{2}\delta_0)\Vol(B(x,e^{-C_4(\bar{t}_k-t)}r))\nonumber\\
&\geq&e^{-2nC_4(\epsilon_0r_\rho)^2}(1-\frac{1}{2}\delta_0)\Vol(B(x,r))\nonumber\\
&\geq&(1-\delta_0)\Vol(B(x,r))\nonumber
\end{eqnarray}
for all
\begin{equation}\nonumber
-(\epsilon_0r_\rho)^2\leq t\leq 0,\hspace{0.3cm}B_{h_k(t)}(x,r)\subset B_{h_k(t)}(x_0,\frac{1}{2}r_\rho).
\end{equation}
By Theorem \ref{t41}, the curvature of $h_k(t)$ satisfies the
bounded
\begin{equation}
|Rm(h_k(t))|(x_0)\leq\frac{1}{4}(\epsilon_0r_\rho)^2,\hspace{0.3cm}\forall x_0\in\psi_k(K_\rho),t\in[-\frac{1}{2}(\epsilon_0r_\rho)^2,0]
\end{equation}
whenever $k$ is large enough. Now recall that
$h_k(t)=\phi_{k,t}^*g_k(t)$ for a family of diffeomorphisms
$\phi_{k,t}$ whose variation $\frac{1}{2}\nabla u(t_k+t)$ is
uniformly bounded. Thus, the curvature of $g_k(t)$ satisfies
\begin{equation}\label{e416}
|Rm(g_k(t))|(x_0)\leq\frac{1}{4}(\epsilon_0r_\rho)^2,\hspace{0.3cm}\forall x_0\in\psi_k(K_{\frac{\rho}{2}}),t\in[-\delta(\epsilon_0r_\rho)^2,0]
\end{equation}
for some uniform $\delta>0$ and any $k$ large enough. The
derivative estimate for the curvature tensor follows directly from
Shi's gradient estimate \cite{Shi}, see also \cite{Ha95b}. The
smooth convergence of $g_k$ follows directly from Hamilton's
compactness theorem for Ricci flow \cite{Ha95a}.

The proof of the claim is completed.
\end{proof}

%Since $u(t_k)$ satisfies the uniform elliptic equation $R(g_k)+\triangle u(t_k)=n$, the $C^\infty$ bound and convergence hold on $\mathcal{R}$.

Finally we show that $g_\infty$ satisfies the K\"ahler-Ricci
soliton equation. We shall apply Perelman's monotonicity of
$\mu(g(t),\frac{1}{2})$ under the K\"ahler-Ricci flow $g(t)$. The
$L^\infty$ estimate to the minimizer of $\mu(g(t),\frac{1}{2})$ will
play an essential role.

Let $f_k$ be a normalized minimizer of $\mu(g_k,\frac{1}{2})$ such that
\begin{equation}\label{e417}
\int e^{-f_k}dv_{g_k}=(2\pi)^n.
\end{equation}
One verifies easily the following variation identity for $f_k$, cf. \cite{Se},
\begin{equation}\label{e418}
2\triangle f_k-|\nabla f_k|^2+f_k=\mu(g_k,\frac{1}{2})-R+2n.
\end{equation}
Denote $v_k=e^{-f_k/2}$, then (\ref{e418}) is equivalent to
\begin{equation}\label{e419}
\triangle v_k=\frac{v_k}{4}\big(R(g_k)-2n-\mu(g_k,\frac{1}{2})-\log v_k\big).
\end{equation}

\begin{lemma}\label{l41}
There is a constant $C_7$ independent of $k$ such that
\begin{equation}\label{e420}
\inf_M f_k\geq -C_7.
\end{equation}
\end{lemma}
\begin{proof}
It suffices to prove a uniform upper bound of $v_k$. It is pointed
out by Rothaus \cite{Ro} that, $\sup_M v_k$ admits a universal upper
bound by an easy iteration argument, using that $R(g_k)$ and
$\mu(g_k,\frac{1}{2})$ are both uniformly bounded. We just mention
that, according to the independent work of Ye \cite{Ye} and Zhang
\cite{ZhQ}, the Sobolev constant under the K\"ahler-Ricci flow has a
uniform bound. In particular,
\begin{equation}\label{e421}
\big(\int_M\phi^{\frac{2n}{n-1}}dv_{g_k}\big)^{\frac{n-1}{n}}\leq C_8\int_M(|\nabla\phi|^2+\phi^2)dv_{g_k},\hspace{0.3cm}\forall\phi\in C^\infty(M),
\end{equation}
for a uniform constant $C_8$. Thus, the iteration process works
uniformly for all $k$. We will not give the details here.
\end{proof}

%It remains to show that $v_k(t)$ has a uniform upper bound for all $t\in[-1,0]$. This follows easily by a maximal principle argument to the backward heat equation
%\begin{equation}\label{e420}
%\frac{\partial}{\partial t}v_k^2=-\triangle v_k^2+(R-n)v_k^2.
%\end{equation}
%One should use the uniform bound of $R$ in (\ref{e46}) once again.

The uniform upper bound of $f_k$ is more interesting. The following
proof relies on the Sobolev inequality and the Poincar\'e inequality
under the K\"ahler-Ricci flow. The first step is to show an $L^2$
estimate of $f_k$.

\begin{lemma}\label{l42}
There exists $C_9$ independent of $k$ such that
\begin{equation}\label{e422}
\int_Mf_k^2dv_{g_k}\leq C_9.
\end{equation}
\end{lemma}
\begin{proof}
Integrating the identity (\ref{e418}) over $M$ and by the bound of $\mu(g_k,\frac{1}{2})$ and $R$ from (\ref{e46}) we get the estimate
\begin{eqnarray}
\int_M|\nabla f_k|^2dv_{g_k}&=&\int_M(f_k+R-\mu(g_k,\frac{1}{2})-2n)dv_{g_k}\nonumber\\
&\leq&\int_Mf_kdv_{g_k}+2C_1\V.\label{e423}
\end{eqnarray}

Another observation is that from the normalization $\int_Me^{-f_k}dv_{g_k}=(2\pi)^n$,
\begin{equation}\label{e424}
\Vol_{g_k}(\{f_k\leq\log \V\})\geq e^{-C_7}
\end{equation}
where $C_7$ is the constant in (\ref{e420}).

Denote the unit measure $d\mu_k=(2\pi)^{-n}e^{-u_k}dv_{g_k}$ where
$u_k=u(t_k)$ is the normalized Ricci potential of $g_k$. Recall the
Poincar\'e inequality on a K\"ahler manifold which satisfies
(\ref{e41}), cf. \cite{Fu},
\begin{equation}\label{e425}
\int_M\phi^2d\mu_k\leq\int_M|\nabla\phi|^2d\mu_k+\big(\int_M\phi d\mu_k\big)^2,\hspace{0.3cm}\forall \phi\in C^\infty(M).
\end{equation}
Set $A=\max(\log \V,C_7)$. Substituting $f_k$ into the Poincar\'e inequality gives the estimate
\begin{eqnarray}
\int_M|\nabla f_k|^2d\mu_k&\geq&\int_Mf_k^2d\mu_k-\big(\int_M|f_k|d\mu_k\big)^2\nonumber\\
&\geq&\int_Mf_k^2d\mu_k-\big(\int_{\{f_k>A\}}|f_k|d\mu_k+A\V\big)^2\nonumber\\
&\geq&\int_Mf_k^2d\mu_k-\big(\int_{\{f_k>A\}}|f_k|d\mu_k\big)^2-2A\V\int_M|f_k|d\mu_k-A^2\V^2\label{e426}.
\end{eqnarray}
By Schwarz inequality
\begin{equation}\label{e427}
\big(\int_{\{f_k>A\}}|f_k|d\mu_k\big)^2\leq\int_{\{f_k>A\}}f_k^2d\mu_k\cdot\int_{\{f_k>A\}}d\mu_k
\end{equation}
where $\int_{\{f_k>A\}}d\mu_k$ can be estimated as follows
\begin{eqnarray}
\int_{\{f_k>A\}}d\mu_k&=&1-\int_{\{f\leq A\}}(2\pi)^{-n}e^{-u_k}dv_{g_k}\nonumber\\
&\leq& 1-(2\pi)^{-n}e^{-C_1}\Vol_{g_k}(\{f_k\leq A\})\nonumber\\
&\leq&1-(2\pi)^{-n}e^{-C_1-C_7}.\label{e428}
\end{eqnarray}
Plugging the estimate of $\int_M|\nabla f_k|^2d\mu_k$ into (\ref{e426}) gives
\begin{equation}\label{e429}
\int_M|\nabla f_k|^2d\mu_k\geq(2\pi)^{-n}e^{-C_1-C_7}\int_Mf_k^2d\mu_k-2A\V\int_M|f_k|d\mu_k-A^2\V^2.
\end{equation}

Let $C_{10}=(2\pi)^ne^{C_1+C_7}$. Then, combining with (\ref{e425}) and (\ref{e429}) yields
\begin{eqnarray}
\int_Mf_k^2dv_{g_k}&\leq&C_{10}\int_M|\nabla f_k|^2dv_{g_k}+2A\V C_{10}\int_M|f_k|dv_{g_k}+A^2\V^2C_{10}\nonumber\\
&\leq&C_{10}(2A\V+1)\int_M|f_k|dv_{g_k}+C_{10}(A^2\V^2+2C_1\V)\nonumber\\
&\leq&C_{10}(2A\V+1)\V^{1/2}\big(\int_Mf_k^2dv_{g_k}\big)^{1/2}+C_{10}(A^2\V^2+2C_1\V)\label{e430}.
\end{eqnarray}
One can derive easily a uniform upper bound of $\int_Mf_k^2dv_{g_k}$.
\end{proof}

The $L^\infty$ bound of $f_k$ follows from a standard Moser's iteration argument.

\begin{lemma}\label{l43}
There exists $C_{11}$ independent of $k$ such that
\begin{equation}\label{e432}
\sup f_k\leq C_{11}.
\end{equation}
\end{lemma}
\begin{proof}
Let $\tilde{f}_k=f_k+C_7+1$ for $C_7$ in (\ref{e420}). It suffices to show a uniform upper bound of $\tilde{f}_k$. Obviously from (\ref{e422}),
\begin{equation}\label{e433}
2\triangle\tilde{f}_k-|\nabla\tilde{f}_k|^2+\tilde{f}_k=\mu_k-R+2n+C_7+1,
\end{equation}
where $\mu_k=\mu(g_k,\frac{1}{2})$. Notice that $\tilde{f}_k\geq 1$.
Multiplying $\tilde{f}_k^{p-1}$ for any $p>1$ onto above identity
and integrating by parts, we get
\begin{eqnarray}
\int_M(\mu_k-R+2n+C_7+1)\tilde{f}_k^{p-1}dv_{g_k}
&=&\int_M(2\triangle\tilde{f}_k-|\nabla\tilde{f}_k|^2+\tilde{f}_k)\tilde{f}_k^{p-1}dv_{g_k}\nonumber\\
&\leq&\int_M(2\tilde{f}_k^{p-1}\triangle\tilde{f}_k+\tilde{f}_k^p)dv_{g_k}\nonumber\\
&=&-\frac{8(p-1)}{p^2}\int_M|\nabla\tilde{f}^{p/2}|^2dv_{g_k}+\int\tilde{f}_k^pdv_{g_k}\label{e434}.
\end{eqnarray}
Rearranging the terms and using the bound of $\mu_k$ and $R$ by (\ref{e46}),
\begin{equation}\label{e435}
\int_M|\nabla\tilde{f}^{p/2}|^2dv_{g_k}\leq pC_{12}\int_M\tilde{f}_k^pdv_{g_k},\hspace{0.3cm}\forall p\geq 2,
\end{equation}
where $C_{12}$ is a constant independent of $k$ and $p\geq 2$.

Define a sequence of positive numbers $p_i=2\cdot(\frac{n}{n-1})^i$,
$i=0,1,2,\cdots,$ and apply the inequality (\ref{e435}) to each
$p_i$. Applying the Sobolev inequality (\ref{e421}) to $\tilde{f}_k$
we obtain
\begin{eqnarray}
\big(\int_M\tilde{f}_k^{\frac{np_i}{n-1}}dv_{g_k}\big)^{\frac{n-1}{n}}&\leq&C_8\big(\int_M|\nabla\tilde{f}_k^{\frac{p_i}{2}}|^2+\tilde{f}_k^{p_i}\big)dv_{g_k}\nonumber\\
&\leq&p_iC_{13}\int_M\tilde{f}_k^{p_i}dv_{g_k}\label{e436}
\end{eqnarray}
where $C_{13}$ is a uniform constant independent of $k$ and $i$. It implies,
\begin{equation}\label{e437}
\|\tilde{f}_k\|_{L^{p_{i+1}}}\leq p_i^{\frac{1}{p_i}}(2C_{13})^{\frac{1}{p_i}}\|\tilde{f}_k\|_{L^{p_i}},\hspace{0.3cm}\forall i\geq 0.
\end{equation}
An interation argument yields the estimate
\begin{equation}\label{e438}
\|\tilde{f}_k\|_{L^\infty}\leq C_{14}\int_M\tilde{f}_k^2dv_{g_k},
\end{equation}
for a uniform constant $C_{14}$. Combining with the $L^2$ estimate of $f_k$ in above lemma gives the desired upper bound of $\tilde{f}_k$ in this case.
\end{proof}

Let $g_k(t)=g(t_k+t)$ be the sequence of K\"ahler-Ricci flow as before and $f_k(t)$ be the associated solution to the backward heat equation
\begin{equation}\label{e440}
\frac{\partial}{\partial t}f_k=-\triangle f_k+|\nabla f_k|^2-R+n
\end{equation}
with initial value at $0$ the minimizer $f_k$. Then the maximal
principle derives the two-sided bound
\begin{equation}\label{e441}
\sup |f_k(t)|\leq C_{16},\hspace{0.3cm}\forall -1\leq t\leq 0
\end{equation}
for some uniform constant $C_{16}$.

Now we are ready to show that $g_\infty$ is a K\"ahler-Ricci soliton.

\begin{claim}\label{c43}
$g_\infty$ satisfies the K\"ahler-Ricci soliton equation on $\mathcal{R}$.
\end{claim}
\begin{proof}
This follows essentially from the monotonicity of Perelman's
entropy functional $\mu(g(t),\frac{1}{2})$ under the original
K\"ahler-Ricci flow $g(t)$ on $M$.

As before, let $f_k$ be one associated minimizer of
$\mu(g_k,\frac{1}{2})$ and $v_k=e^{-f_k/2}$. Let
$K_\rho=\{x\in\mathcal{R}|d(x,\mathcal{S})\geq\rho\}$ for any
$\rho>0$. Then, by (\ref{e416}) and Shi's local gradient estimate
for curvature \cite{Shi, Ha95b}, the geometry of $\psi_k(K_\rho)$ is
$C^\infty$ uniformly bounded whenever $k$ is large enough. The
elliptic regularity theory to (\ref{e419}) gives the $C^\infty$
bound of $v_k$ as well as $f_k$ on $\psi_k(K_\rho)$. Hence, passing
a subsequence and letting $\rho\rightarrow 0$, $f_k$ will converge
smoothly to a function $f_\infty$ on $\mathcal{R}$.

Another implication of (\ref{e416}) is that by Hamilton's
compactness theorem for Ricci flow, for any fixed $\rho$ there
exists $\tau_\rho>0$ such that, $(\psi_k(K_\rho),g_k(t))$ converges
along a subsequence to a K\"ahler-Ricci flow
$(M_i,g_{\rho,\infty}(t))$ on some time interval $[-\tau_\rho,0]$.
Since
$(\psi_k(K_\rho),g(t_k))\stackrel{C^\infty}{\longrightarrow}(K_\rho,g_\infty)$
by Claim \ref{c42}, we may assume that, up to an isometry,
$M_\rho=K_\rho$ and $g_{\rho,\infty}(0)=g_\infty$ on $K_\rho$.

The remaining parts of the proof are standard, cf. \cite{Se, SeTi}:
Let $f_k(t)$ be the solution to (\ref{e440}) on the time interval
$[-\tau_\rho,0]$. By (\ref{e441}), $f_k$ is $C^0$ uniformly bounded
on $M\times[-\tau_\rho,0]$. The regularity theory for (local) heat
equation (\ref{e440}) gives the bound
\begin{eqnarray}
\|\nabla^lf_k\|_{\psi_k(K_\rho)}\leq C_{l,\rho},\hspace{0.3cm}\forall t\in[-\tau_\rho,0]
\end{eqnarray}
whenever $k$ is large enough, where $C_{l,\rho}$ is a uniform
constant depending only on $l$ and $\rho$. In particular, $f_k(t)$
converges along a subsequence to a family of smooth functions
$f_{\rho,\infty}(t)$ on $K_i\times[-\tau_\rho,0]$ with
$f_\infty(0)=f_\infty|_{K_\rho}$.

By Perelman's monotonicity formula \cite{Pe}, also see \cite{SeTi},
\begin{equation}\nonumber
\frac{d}{dt}\mathcal{W}(g(t),f_k(t),\frac{1}{2})=\int_M\big(|R_{i\bar{j}}+\nabla_i\nabla_{\bar{j}}f_k-g_{i\bar{j}}|^2+2|\nabla_i\nabla_jf_k|^2\big)
(2\pi)^{-n}e^{-f_k}dv_{g(t)}.
\end{equation}
In particular, $\mu(g(t),\frac{1}{2})$ is increasing and
\begin{eqnarray}
&&\int_{t_k-\tau_\rho}^{t_k}\int_M\big(|R_{i\bar{j}}+\nabla_i\nabla_{\bar{j}}f_k-g_{i\bar{j}}|^2+2|\nabla_i\nabla_jf_k|^2\big)(2\pi)^{-n}e^{-f_k}dvdt\nonumber\\
&&\leq\mu(g(t_k),\frac{1}{2})-\mu(g(t_k-\tau_\rho),\frac{1}{2}).\nonumber
\end{eqnarray}
By the monotonicity and boundedness of $\mu(g(t),\frac{1}{2})$,
\begin{equation}\nonumber
\mu(g_k(0),\frac{1}{2})-\mu(g_k(-\tau_\rho),\frac{1}{2})=\mu(g(t_k),\frac{1}{2})-\mu(g(t_k-\tau_\rho),\frac{1}{2})\rightarrow 0,\hspace{0.3cm}\mbox{as }k\rightarrow\infty.
\end{equation}
Passing to the limit space $(K_\rho,g_{\rho,\infty})$, this implies that for any $\rho>0$,
\begin{eqnarray}
&&\int_{-\tau_\rho}^{0}\int_{K_\rho}\big(|Ric+\nabla\bar{\nabla}f_\infty-g_{\rho,\infty}|^2+2|\nabla\nabla f_\infty|^2\big)(2\pi)^{-n}dvdt\nonumber\\
&&\leq e^{C_{16}}\int_{-\tau_\rho}^{0}\int_{K_\rho}\big(|Ric+\nabla\bar{\nabla}f_\infty-g_{\rho,\infty}|^2+2|\nabla\nabla f_\infty|^2\big)(2\pi)^{-n}e^{-f_{\rho,\infty}} dvdt\nonumber\\
&&\leq e^{C_{16}}\lim_{k\rightarrow\infty}\int_{-\tau_\rho}^{0}\int_{\psi_k(K_\rho)}\big(|Ric+\nabla\bar{\nabla}f_k-g_k|^2+2|\nabla\nabla f_k|^2\big)(2\pi)^{-n}e^{-f_k}dvdt\nonumber\\
&&\leq e^{C_{16}}\lim_{k\rightarrow\infty}\big(\mu(g_k(0),\frac{1}{2})-\mu(g_k(-\tau_\rho),\frac{1}{2})\big)\nonumber\\
&&=0.\nonumber
\end{eqnarray}
where $C_{16}$ is the constant in (\ref{e441}). That is, the couple $(g_{\rho,\infty},f_{\rho,\infty})$ satisfies
\begin{equation}\label{e442}
\left\{ \begin{array}{ll}
Ric+\nabla\bar{\nabla}f_{\rho,\infty}=g_{\rho,\infty},\\
\nabla\nabla f_{\rho,\infty}=0,
\end{array} \right.
\end{equation}
on the space time $K_\rho\times[-\tau_\rho,0]$. In particular, at time $0$,
\begin{equation}\label{e443}
\left\{ \begin{array}{ll}
Ric+\nabla\bar{\nabla}f_{\infty}=g_{\infty},\\
\nabla\nabla f_{\infty}=0,
\end{array} \right.
\end{equation}
which means that $g_\infty$ satisfies the shrinking K\"ahler-Ricci
soliton equation on each $K_\rho$ with potential $f_\infty$. The
proof then follows from the arbitrariness of $\rho$.
\end{proof}

Summing up Claims \ref{c42} and \ref{c43} proves our theorem \ref{t40}.

We finally show that the limits of $f_k$ and $u_k$ coincide. By the
$C^1$ bound of $u_k=u(t_k)$, cf. (\ref{e46}), the Ricci potentials
$u_k$ will converge along a subsequence to a Lipschitz function
$u_\infty$ on $Y$. Applying the elliptic regular theory to
$\triangle u(t_k)+R(g_k)=n$ shows that $u_\infty$ is actually smooth
on $\mathcal{R}$. It is clear that $u_\infty$ is the Ricci potential
of $g_\infty$:
\begin{equation}\label{e444}
R_{i\bar{j}}(g_\infty)+\nabla_i\nabla_{\bar{j}}u_\infty=g_{\infty,i\bar{j}}.
\end{equation}
We have the following proposition.

\begin{proposition}[Compare Proposition 14 in \cite{SeTi}]
$f_\infty=u_\infty$ on $\mathcal{R}$. In particular, $f_\infty$ can be extended to be a globally Lipschitz function on $Y$.
\end{proposition}
\begin{proof}
We first show that $f_\infty$ is globally Lipschitz on $\mathcal{R}$
and so admits a natural extention over $Y$. From the K\"ahler-Ricci
soliton equation (\ref{e443}), using second Bianchi identity, one
verifies easily the following identity, cf. \cite{Ha95b} for
example,
\begin{equation}\label{e445}
R(g_\infty)+|\nabla f_\infty|^2=f_\infty+const.,\hspace{0.3cm}\mbox{on }\mathcal{R}.
\end{equation}
We mention that one should apply the connectedness of $\mathcal{R}$
to derive the identity. By Perleman's estimate to $R(g_k)$ and
estimate (\ref{e441}), we have the uniform bound $|R(g_\infty)|\leq
C_1$ and $|f_\infty|\leq C_{16}$ on $\mathcal{R}$ because of the
smooth convergence there. It concludes the uniform bound of $|\nabla
f_\infty|$ on $\mathcal{R}$.

We next show that $\nabla(u_\infty-f_\infty)=0$ on $\mathcal{R}$. Combining with the normalization
\begin{equation}\label{e446}
\int_{\mathcal{R}}e^{-f_\infty}dv_{g_\infty}=\int_{\mathcal{R}}e^{-u_\infty}dv_{g_\infty}=(2\pi)^n
\end{equation}
which follows directly from the limiting process, we conclude that $u_\infty=f_\infty$.

To prove (\ref{e446}), it suffices to show
$\int_{\mathcal{R}}|\nabla(u-f)|^2dv_{g_\infty}=0$.To this aim,
choose a sequence of compact submanifolds $V_i\subset\mathcal{R}$
such that $\cup V_i=\mathcal{R}$ and $\Vol(\partial V_i)\rightarrow
0$ as $i\rightarrow\infty$. This can be done since the boundary
$\partial\mathcal{R}=\mathcal{S}$ has codimension at least 4. Then
take integration by parts,
\begin{eqnarray}
\int_{V_i}|\nabla(u_\infty-f_\infty)|^2dv_{g_\infty}&=&-\int_{\partial V_i}(u_\infty-f_\infty)\langle\nabla(u_\infty-f_\infty),\mu\rangle\nonumber\\
&&-\int_{V_i}(u_\infty-f_\infty)\triangle(u_\infty-f_\infty)dv_{g_\infty}\nonumber\\
&\leq&\sup|(u_\infty-f_\infty)\nabla(u_\infty-f_\infty)|\Vol(\partial V_i)
\end{eqnarray}
which tends to 0 as $i\rightarrow\infty$. Here we used that
$\triangle u_\infty=\triangle f_\infty$ by equations (\ref{e442})
and (\ref{e443}). The proof is now completed.
\end{proof}

We end the paper with several remarks.

\begin{remark}
By the arguments in \cite{Se}, one can show that the metric $d$
(respectively $g_\infty$) can be extended to be a family of metrics
$d_t$, $t\in(-\infty,\infty)$, with $d_0=d$ (respectively
$g_\infty(t)$ with $g_\infty(0)=g_\infty$) such that
$(M,g(t_k+t))\stackrel{d_{GH}}{\longrightarrow}(Y,d_t)$
(respectively
$\psi_k^*g(t_k+t)\stackrel{C^\infty}{\longrightarrow}g_\infty(t)$)
at each time $t$.
\end{remark}

\begin{remark}
It should be true that the limit $g_\infty(t)$ is independent of the choice of the sequence $t_k$.
\end{remark}

\begin{remark}
Let $(M,g(t))$ be a K\"ahler-Ricci flow without assumption
(\ref{e42}) in a prior. Then $(M,g(t))$ converges subsequentially to
a compact metric space $(Y,d)$. Suppose $N\subset Y$ is a smooth
subset on which $d$ is induced by a smooth metric $g_\infty$. If
$g(t)\stackrel{C^\infty}{\longrightarrow}g_\infty$, then $g_\infty$
satisfies a shrinking K\"ahler-Ricci soliton equation on $N$.
\end{remark}

\end{document}